\documentclass[notitlepage,11pt,a4paper,reqno]{amsart}
\usepackage[foot]{amsaddr}
\usepackage{amsmath,amssymb, amsbsy}
\usepackage{color,psfrag}
\usepackage[dvips]{graphicx}
\usepackage{color}
\theoremstyle{plain}

\newtheorem{theorem}{Theorem}[section]

\newtheorem{lemma}{Lemma}[section]
\newtheorem{corollary}{Corollary}[section]
\newtheorem{remark}{Remark}[section]
\newtheorem{definition}{Definition}[section]

\newcommand{\eps}{\varepsilon}

\newcommand{\authorfootnotes}{\renewcommand\thefootnote{\@fnsymbol\c@footnote}}%

\newcommand{\be} {\begin{equation}}
\newcommand{\ee} {\end{equation}}
\newcommand{\bea} {\begin{eqnarray}}
\newcommand{\eea} {\end{eqnarray}}
\newcommand{\Bea} {\begin{eqnarray*}}
\newcommand{\Eea} {\end{eqnarray*}}

\newcommand{\de} {\delta}

\newcommand{\ga} {\gamma}
\newcommand{\Ga} {\Gamma}

\newcommand{\De}{\Delta}

\newcommand{\no} {\nonumber}

\newcommand{\lab} {\label}

\newcommand{\Rn}{\mathbb R^N}

\catcode`\@=11

\numberwithin{equation}{section} \allowdisplaybreaks

\usepackage[text={6in,8.6in},centering]{geometry}
\begin{document}
    
    \title[Integral representation for fractional Hardy equation]{Integral representation using Green function for fractional Hardy equation}

\author[Mousomi Bhakta]{Mousomi Bhakta$^\dag$}
\address{$^\dag$ Department of Mathematics, Indian Institute of Science Education and Research, Dr. Homi Bhaba Road, Pune-411008, India}
\email{mousomi@iiserpune.ac.in}

\author[Anup Biswas]{Anup Biswas$^\dag$}
\email{anup@iiserpune.ac.in}

\author[Debdip Ganguly]{Debdip Ganguly$^\dag$}
\email{debdip@iiserpune.ac.in}

\author[Luigi Montoro]{Luigi Montoro$^\ddag$}
\address{$^\ddag$ Dipartimento di Matematica e Informatica, Unical, Ponte Pietro Bucci 31 B, 87036 Arcavacata
di Rende, Cosenza, Italy.}
\email{montoro@mat.unical.it}

\date{}

\maketitle

\begin{abstract}
Our main aim is to study Green function for the fractional Hardy operator 
 $P:=(-\De)^s -\frac{\theta}{|x|^{2s}}$ in $\Rn$, where $0<\theta<\Lambda_{N,s}$ and $\Lambda_{N,s}$ is the best constant in the fractional Hardy inequality. Using Green function, we also show that the integral representation of the weak solution holds.
\end{abstract} 

\section{Introduction}
The aim of this work  is to show that the potential theoretic solution for fractional Hardy equation coincides with the weak solution and it admits an integral representation  using the  Green function of the fractional Hardy operator. There is a wide literature regarding problems involving the fractional Hardy potential. Avoiding to disclose the discussion  we refer to the following (far from being complete) list  of works and  references therein \cite{AABP, AADP, AMP, AMP1, BDK, GGJP,  DMPS, FF1, FF, FLS-1}.

In this work, we consider the following fractional Hardy equation
\begin{equation}\label{main-semi}
  \tag{$\mathcal P$}
\left\{\begin{aligned}
 (-\De)^s u -\theta\frac{u}{|x|^{2s}}&= \varphi(x) \quad\text{in }\,\,\Rn \\
            u &\in \dot{H}^s(\Rn), \no
 \end{aligned}
  \right.
\end{equation}
where $s\in(0,1)$ is fixed, $N>2s$,  $\varphi\in C_c(\Rn)$ and $(-\De)^s$ denotes the  fractional Laplace operator which can be defined for the Schwartz class functions $\mathcal{S}(\Rn)$  as follows:
\begin{equation} \label{De-u}
  \left(-\Delta\right)^su(x): = c_{N,s} 
\, \text{P.V.} \int_{\Rn}\frac{u(x)-u(y)}{|x-y|^{N+2s}} \, {\rm d}y, \quad c_{N,s}= \frac{4^s\Ga(N/2+ s)}{\pi^{N/2}|\Ga(-s)|}.
\end{equation}
Here $\theta\in(0, \Lambda_{N,s})$, where 
$$\Lambda_{N,s}:=2^{2s}\frac{\Gamma^2(\frac{N+2s}{4})}{\Gamma^2(\frac{N-2s}{4})}$$ denotes the sharp constant in the Hardy inequality 
$$\Lambda_{N,s} \, \int_{\Rn} \, \frac{|u(x)|^2}{|x|^{2s}} \, {\rm d}x \leq \int_{\Rn}|\xi|^{2s}|\mathcal{F}(u)(\xi)|^2 \, {\rm d}\xi.$$ In above inequality $\mathcal{F}(u)$ denotes the Fourier transform of $u$. 

It is well-known that (see \cite{GGJP}, \cite{FLS-1}) if  $0<\theta\leq\Lambda_{N,s}$, then there exists a unique number $\ga$ such that \be\lab{psi-theta}
0<\ga\leq\frac{N-2s}{2} \quad\text{and}\quad
\theta=2^{2s}\frac{\Gamma(\frac{\ga+2s}{2})\Gamma(\frac{N-\ga}{2})}{\Gamma(\frac{N-\ga-2s}{2})\Gamma(\frac{\ga}{2})}.\ee

Further, $\theta=\Lambda_{N,s}$ implies $\ga=\frac{N-2s}{2}$. Therefore in this paper we consider the range 
$$0<\ga<\frac{N-2s}{2}.$$
This constant $\ga$ plays an important role in the analysis of fractional Hardy equations (see \cite{FLS-1}) and it will appear in several equations below.

\begin{definition}
Let $s\in(0,1)$. We define the homogeneous fractional Sobolev space of order $s$
as $$\dot{H}^s(\Rn):=\bigg\{u\in L^{2^*_s}(\Rn) \,:\, \int_{\Rn}|\xi|^{2s}|\mathcal{F}(u)|^2 \, {\rm d}\xi<\infty  \bigg\},\quad \text{where}\;\; 2^*_s= \frac{2N}{N-2s}.$$
In particular, $\dot{H}^s(\Rn)$ is  the completion of $C^\infty_c(\Rn)$ under the norm 
\be\lab{dot-H}\|u\|_{\dot{H}^s(\Rn)}^2:=\int_{\Rn}|\xi|^{2s}|\mathcal{F}(u)|^2 \, {\rm d}\xi.\ee
\end{definition}
It is also known that (see \cite{FLS-1}) for $s\in (0,1)$, $N\geq 1$, it holds
$$\int_{\Rn}|\xi|^{2s}|\mathcal{F}(u)|^2 \, {\rm d}\xi=\frac{c_{N,s}}{2}
\int_{\Rn}\int_{\Rn}\frac{|u(x)-u(y)|^2}{|x-y|^{N+2s}} \, {\rm d}x \, {\rm d}y, \quad \forall \, u\in \dot{H}^s(\Rn).$$

\begin{definition}
We say that $u\in \dot{H}^s(\Rn)$ is a weak solution of \eqref{main-semi} if for every $\psi \in C_c^{\infty}(\mathbb{R}^N)$ there holds 
$$
\frac{c_{N,s}}{2}
\int_{\Rn}\int_{\Rn}\frac{(u(x)-u(y))(\psi(x)-\psi(y))}{|x-y|^{N+2s}}\, {\rm d}x\, {\rm d}y\, - \,  \theta \, \int_{\Rn} \frac{u \psi}{|x|^{2s}} \, {\rm d}x= \int_{\Rn} \varphi(x) \, \psi(x) \, {\rm d}x.
$$
\end{definition}

\begin{remark}\lab{r:1}
{\rm 
Since $\theta<\Lambda_{N,s}$ implies, norm in \eqref{dot-H} is equivalent to
$$\bigg(\frac{c_{N,s}}{2}\int_{\Rn}\int_{\Rn}\frac{|u(x)-u(y)|^2}{|x-y|^{N+2s}}\, {\rm d}x\, {\rm d}y\, - \,  \theta \, \int_{\Rn} \frac{|u|^2}{|x|^{2s}}\, {\rm d}x \bigg)^\frac{1}{2},$$ it is easy to see that 
the weak solution of  \eqref{main-semi} is unique.  
}
\end{remark}

\medskip 

\begin{definition}
Let $s\in(0,1)$. We define the fractional Sobolev space of order $s$
as $$ H^s(\Rn):=\bigg\{u\in L^{2}(\Rn) \,:\, \int_{\Rn}\int_{\Rn}\frac{|u(x)-u(y)|^2}{|x-y|^{N+2s}} \, {\rm d}x \, {\rm d}y<\infty\bigg\},$$
and 
$$\|u\|_{{H}^s(\Rn)}^2:=\|u\|_{L^2(\Rn)}^2+ \int_{\Rn}\int_{\Rn}\frac{|u(x)-u(y)|^2}{|x-y|^{N+2s}} \, {\rm d}x \, {\rm d}y.$$
\end{definition}

As mentioned above we would be interested in the operator
$$P:=(-\De)^s u -\theta\frac{u}{|x|^{2s}}, \quad 0<\theta<\Lambda_{N,s}.$$
Recently, it is shown in \cite{GGJP} that for $0<\theta\leq\Lambda_{N,s}$, $-P$ generates 
a strongly continuous contraction semigroup $\{\tilde S(t)\}$ in $L^2(\Rn)$. Furthermore,
heat kernel of $-P$ also exists in $\Rn\setminus\{0\}$ i.e., there exists a 
 nonnegative function $p(t, x,y)$ such that
\begin{equation}\label{eq:13-06-2}
\tilde{S}(t) g=\int_{\Rn} p(t, x, y) g(y) \,  {\rm d}y, \quad t>0,\quad g\in L^2(\Rn).
\end{equation}
see \cite{GGJP} for more details.
%
%
%
%

\begin{definition}\label{def:Green}
We say that a function  
$$G: \Rn\setminus\{0\}\times \Rn\setminus\{0\}\longrightarrow [-\infty, \infty]$$ is a Green function for the operator $P$, if 
\begin{itemize}
\item[$(i)$]
$$G(x,y)=G(y,x), \quad\forall\, y,x\neq 0;$$  
\item[$(ii)$] for $x\neq 0$ , $$ G(x,.)\in L^1_{loc}(\Rn), \quad\text{and}\quad \int_{\Rn}\frac{|G(x,y)|}{1+|y|^{N+2s}} \, {\rm d}y<\infty;$$
\item[$(iii)$] for any  $x\neq 0$, it holds
$$P G_P(x, \cdot)=\delta_x\quad \text{in}\quad \mathbb{R}^N,$$
in the sense of distribution.
\end{itemize}
\end{definition}

As usual in potential theory, we define (see also Lemma \ref{l:heat-ker} below) the  function $G_P(x, y)$ as follows
\be\lab{Green}G_P(x, y):=\int_0^\infty p(t, x, y) \, {\rm d}t, \quad x,y\in\Rn\setminus\{0\},\,\, x\neq y.\ee

It is also known that heat kernel is symmetric i.e.,  $p(t, y, z)=p(t,z,y)$ (see \cite[pg 6]{GGJP}) and this implies  $G_P(x, y)=G_P(y, x)$.

\begin{remark}
It is also common in literature (see \cite[Appendix D]{GR}) to define Green function using integral representation. More precisely, a symmetric $G: \Rn\setminus\{0\}\times \Rn\setminus\{0\}\longrightarrow [-\infty, \infty]$ is said to be Green function
of $P$ if 
\begin{itemize}
\item [$(i)$] $G(x,\cdot)\in L^1_{loc}(\Rn)$  a.e. $x$;
\item [$(ii)$] for all $\varphi \in C_c(\Rn)$, the function 
\begin{equation}\nonumber
x\longrightarrow  \int_{\Rn}   G(x, y)\varphi (y) \, {\rm d}y, \quad  \text{a.e.}\,\, x,
\end{equation}
belongs to $\dot H^s(\mathbb R^N)$ and it is the unique weak solution to the equation \eqref{main-semi}.
\end{itemize}
We show in Theorem~\ref{t:int form} that $G_P$ in \eqref{Green} does satisfy these conditions.
\end{remark}

\vspace{2mm}

Below we state the main results of this paper.  
\begin{theorem}\lab{t:int form}
Let $0<\theta<\Lambda_{N,s}$, $\varphi \in C_c(\Rn)$ and 
\begin{equation}\nonumber\psi (x) := \int_{\Rn}   G_P(x, y)\varphi (y) \, {\rm d}y, \quad x\neq 0,\end{equation}  where $G_P$ is as defined in \eqref{Green}.
Then $\psi\in \dot{H}^s(\Rn)$ and $P  \psi = \varphi$ in $\Rn$ holds in the weak sense.
\end{theorem}

Moreover, by the previous result, we are able to show a representation formula for the weak solution to \eqref{main-semi}. We have
\begin{corollary}\label{thm-1}
Let $0<\theta<\Lambda_{N,s}$ and  $u$ be a weak solution of the equation \eqref{main-semi}. Then there holds
$$u(x) = \int_{\Rn} G_{P}(x, y) \varphi(y) \, {\rm d}y, \quad x\neq 0,$$
where $G_P$ is as in Theorem \ref{t:int form}. 
\end{corollary}

\begin{theorem}\lab{l:fund-sol}
Let  $0<\theta<\Lambda_{N,s}$, $x\neq 0$ and $G_P$ be defined as in \eqref{Green}. Then it holds that
$$P G_P(x, \cdot)=\delta_x\quad \text{in}\quad \mathbb{R}^N,$$
in the sense of distribution.
\end{theorem}

Thanks to Theorem \ref{l:fund-sol} the function $G_P$ defined in \eqref{Green} is actually a Green function in the meaning of Definition \ref{def:Green}. Next we show that the integral representation of solution of \eqref{main-semi} holds for more general class of functions $\varphi$, namely we prove the following:

\begin{theorem}\lab{t:int-gen}
Let the function $\varphi$ in \eqref{main-semi} be such that 
$0\leq \varphi\in L^\frac{2N}{N+2s}(\Rn)$ and $\varphi$ be lower semi-continuous. Then there exists a unique weak solution to \eqref{main-semi}.

Further, if 
\begin{equation}\label{eq:27thJune}
\displaystyle\int_{\Rn}G_P(x,y)\varphi(y)\, {\rm d}y<\infty,
\end{equation}
then the solution (denote it by $\psi$) will be of the form $$\dot{H}^s(\Rn)\ni\psi(x)=\int_{\Rn}G_P(x,y)\varphi(y) \, {\rm d}y, \quad x\neq 0.$$ 
\end{theorem}

In our next theorem, we prove the \textit{uniqueness} of Green function.

\begin{theorem}\lab{t:unique}
Let $\Phi\geq 0$ be such that for $x\neq 0$,
$$ \Phi(x,.)\in L^1_{loc}(\Rn) \quad\text{and}\quad \int_{\Rn}\frac{|\Phi(x,y)|}{1+|y|^{N+2s}} 
\, {\rm d}y<\infty$$
and it holds
$$P \Phi(x, \cdot)=\delta_x\quad \text{in}\quad \mathbb{R}^N, \quad x\neq 0,$$
in the sense of distribution. Further, if  $\Phi(x,y)=\Phi(y,x)$,  
$\Phi(x,y)$ is continuous for  $x,y\neq 0$, $y\neq x$ and
$$\psi_{\varphi}:=\int_{\Rn}\Phi(x,y)\varphi(y) \, {\rm d}y\in\dot{H}^s(\Rn)\quad\forall\quad 0\leq \varphi\in C_c^\infty(\Rn),$$
then $\psi_{\varphi}$ satisfies \eqref{main-semi} and $\Phi=G_P$ in $(\Rn\setminus\{0\})\times (\Rn\setminus\{0\})$.
\end{theorem}

 It is well known that Green functions are closely related to fundamental solutions.
For $\theta=0$, using Fourier transform, it can be shown  (see \cite[Theorem 2.3]{B}) that 
\begin{equation*}
\Phi_F(x)=\left\{
\begin{array}{lll}
a(N, s) \, |x|^{-N+2s} & \text{for} \; n\neq 2s,
\\
a(N, s) \, \log |x| & \text{for} \; n= 2s,
\end{array}
\right .
\end{equation*}
where $a(N, s)$ is a suitable normalizing constant, 
is the fundamental solution of $(-\De)^s$ in~$\Rn$. Then the Green function is given by $\Phi(x, y)=\Phi_F(x-y)$.
Fall in \cite[Lemma 4.1]{FF1} (take $\alpha=\ga-\frac{N-2s}{2}$ in \cite{FF1}) proved that 
$$\tilde\Phi(x)=\frac{1}{|x|^{N-2s-\ga}},$$ 
solves  $P\tilde\Phi=0$ in $\Rn\setminus\{0\}$. This seems some what related to fundamental solution. 
The main difficulty is to guess the right candidate for the Green function. Though the heat-kernel of fractional Hardy operator is obtained recently in \cite{GGJP},
we neither know any regularity property of the heat-kernel nor if it solves the corresponding heat equation. This is another difficulty in 
considering the standard approach of (local PDE) in getting the Green function.
To obtain the Green function for $P$ we take the approach of semigroup theory which is also closely  related to the criticality theory. 


It's well-known that the integral representation of solution using Green function of the operator is very useful in studying various properties of solutions. In the case of $\theta=0$, i.e.,  for the fractional Laplace operator, integral representation of solution using Green function were used in many papers, to cite few we mention  
\cite{CDM}, \cite{CLO}, \cite{FLS}, \cite{S}, \cite{Lei}. In our forthcoming paper \cite{BGM}, we establish an asymptotic behaviour of solution for the fractional Hardy equation using the integral representation of the solution.

\medskip

{\bf Notation}: In this paper $C, \tilde C, \cdots$ denote the generic constant which may vary from line to line. We denote by $B_R(x)$ the ball cantered at $x$ and of radius $R$. By $C_c(\Rn)$ and $C_c^\infty(\Rn)$ we denote the continuous functions with compact support and $C^\infty$ functions with compact support respectively. $\de_x$ denotes the Dirac distribution centered at $x$. 

\section{Integral representation solves the equation}
In this Section we aim to prove Theorem \ref{t:int form}--Theorem \ref{t:unique}. Towards this, first we establish two sided estimate for the  Green function of fractional Hardy operator~$P$.  In context to the local case, recently in \cite{MNS}, the authors obtain sharp two sided Green function estimate for second-order elliptic operator of Fuchsian-type on $\Rn.$

\subsection{Estimate for Green function}
In the following, for two   given non negative functions $f_1,f_2$, by $f_1 \approx f_2$, we mean that there exists some positive constant $C$ such that 
$C^{-1}{f_1}\leq f_2\leq Cf_1$.  Now we prove the following 
\begin{lemma}\lab{l:heat-ker}
Let $0<\theta<\Lambda_{N,s}$ and $\ga$ the unique solution to     \eqref{psi-theta}. Then the Green function~$G_P$ given in \eqref{Green} 
is well defined. Moreover, 
\begin{equation}\nonumber
G_{P}(x, y) \approx |x-y|^{-(N-2s-2\ga)}(|x-y|^{-\ga}+|x|^{-\ga})(|x-y|^{-\ga}+|y|^{-\ga}),
\end{equation}
for  $x,y\in\Rn\setminus\{0\}$, $x\not=y$.
\end{lemma}
\begin{proof}
From \cite[Theorem 1.1]{GGJP}, it follows that heat kernel $p$ of the Hardy operator in $\Rn \setminus \{ 0\}$ satisfies
\be\lab{heat-ker}
p(t,x,y)\approx (1+t^\frac{\ga}{2s}\,|x|^{-\ga})(1+t^\frac{\ga}{2s}\,|y|^{-\ga})\bigg(t^{-\frac{N}{2s}}\wedge \frac{t}{|x-y|^{N+2s}}\bigg), \quad x,y\not=0,\, t>0.
\ee
We claim that,  for $0<\theta<\Lambda_{N,s}$\begin{eqnarray}\nonumber
&&\int_0^\infty p(t,x, y) \, {\rm d}t\\\nonumber
&&\approx\bigg(|x-y|^{-(N-2s)}+(|x|^{-\ga}+|y|^{-\ga})|x-y|^{-(N-2s-\ga)}+|x|^{-\ga}|y|^{-\ga}|x-y|^{-(N-2s-2\ga)}\bigg),
\end{eqnarray}
for all $x,y\in\Rn\setminus\{0\}$, $x\not=y$. 
To see this, first of all note that 
$$t^{-\frac{N}{2s}}\geq \frac{t}{|x-y|^{N+2s}} \,\,\Longleftrightarrow \,\,t\leq |x-y|^{2s}.$$ Therefore
\begin{eqnarray}\lab{24-4-1}
\int_0^\infty p(t,x, y)\, {\rm d}t&\approx&\int_{0}^{|x-y|^{2s}}(1+t^\frac{\ga}{2s}\,|x|^{-\ga})(1+t^\frac{\ga}{2s}\,|y|^{-\ga})\frac{t}{|x-y|^{N+2s}} \, {\rm d}t\no\\
&&\quad+\int_{|x-y|^{2s}}^{\infty}(1+t^\frac{\ga}{2s}\,|x|^{-\ga})(1+t^\frac{\ga}{2s}\,|y|^{-\ga})t^{-\frac{N}{2s}}\, {\rm d}t\no\\
&:=&I_1+I_2 .
\end{eqnarray}
We deduce that
\begin{eqnarray}\lab{24-4-2}
I_1 &=& \frac{1}{|x-y|^{N+2s}}\int_{0}^{|x-y|^{2s}}t\,{\rm d}t +\frac{(|x|^{-\ga}+|y|^{-\ga})}{|x-y|^{N+2s}}\int_{0}^{|x-y|^{2s}}t^{\frac{\ga}{2s}+1}\,{\rm d}t\no\\
&&\quad+\frac{|x|^{-\ga}|y|^{-\ga}}{|x-y|^{N+2s}}\int_{0}^{|x-y|^{2s}}t^{\frac{\ga}{s}+1}\, {\rm d}t\no\\
&=&C(s,
\gamma)\bigg(\frac{1}{|x-y|^{N-2s}}+\frac{|x|^{-\ga}+|y|^{-\ga}}{|x-y|^{N-2s-\ga}}+\frac{|x|^{-\ga}|y|^{-\ga}}{|x-y|^{N-2s-2\ga}}\bigg).
\end{eqnarray}
Analogously, since $N-2s>2\ga$, we obtain
\begin{eqnarray}\lab{24-4-3}
I_2 &=& \int_{|x-y|^{2s}}^\infty t^{-\frac{N}{2s}}\,{\rm d}t+(|x|^{-\ga}+|y|^{-\ga})\int_{|x-y|^{2s}}^\infty t^{-\frac{N-\ga}{2s}}\, {\rm d}t\no\\
&&\quad+|x|^{-\ga}|y|^{-\ga}\int_{|x-y|^{2s}}^\infty t^{-\frac{N-2\ga}{2s}}\,{\rm d}t\no\\
&=&C(s,N,\gamma)\bigg(\frac{1}{|x-y|^{N-2s}}+\frac{|x|^{-\ga}+|y|^{-\ga}}{|x-y|^{N-2s-\ga}}+\frac{|x|^{-\ga}|y|^{-\ga}}{|x-y|^{N-2s-2\ga}}\bigg).
\end{eqnarray}
Finally, combining \eqref{24-4-2} and \eqref{24-4-3} with \eqref{24-4-1}  (and \eqref{Green}) we get the conclusion.
\end{proof}

\subsection{Semigroup and quadratic forms associated with $P$}

Let $\tilde{S}(t):= e^{-tP}$ be the  semigroup generated by $-P$ i.e.,
$$\mathrm{Dom}(-P)=\bigg\{f\in L^2(\Rn)\; :\; \lim_{t\to 0+} \frac{1}{t} (\tilde{S}(t)-1)f
\; \text{exists in}\; L^2(\Rn)\bigg\},$$
and
\be\lab{3-6-1} -P f := \lim_{t\to 0+} \frac{1}{t} (\tilde{S}(t)-1)f.\quad (f \in \mathrm{Dom}(-P))\ee

\begin{lemma}\label{L2-hardy}
Let $\varphi \in C_c(\Rn)$ and  $\tilde{S}(t)$ be defined as above. For $\alpha>0$,
define $u_{\alpha}$ by
\begin{equation}\label{eq:13-06-1}
u_{\alpha} : = \int_{0}^{\infty} e^{-\alpha \tau} \, \tilde{S}(\tau) \, \varphi\, {\rm d}\tau, \quad x\neq 0.
\end{equation}
Then there holds

$$
\int_{\Rn} \frac{u_{\alpha}^2(x)}{|x|^{2s}} \, {\rm d}x < \infty.
$$
\end{lemma}

\begin{proof}
Let $K := \mathrm{supp}(\varphi).$ We note  that $u_{\alpha}$ given in \eqref{eq:13-06-1}  is well defined since $\{\tilde S(t)\}$ is a contraction semigroup. Moreover, using  \eqref{eq:13-06-2} we get
\begin{equation}\lab{28-5-5}
u_{\alpha}(x) = \int_{0}^{\infty} e^{-\alpha \tau} \, \tilde{S}(\tau) \, \varphi\, {\rm d}\tau \leq \int_{\Rn} G_{P}(x, y) \,| \varphi(y)| \, {\rm d}y. 
\end{equation}
Fix  $R >>1$ be such that  $K \subset\subset B_{{R}/{2}}(0).$ We write 

$$
\int_{\Rn} \frac{u_{\alpha}^2(x)}{|x|^{2s}} \, {\rm d}x 
= \int_{B_R(0)}\frac{u_{\alpha}^2(x)}{|x|^{2s}} \, {\rm d}x 
+ \int_{\Rn\setminus B_R(0)} \frac{u_{\alpha}^2(x)}{|x|^{2s}} \, {\rm d}x.
$$
Let $ x \in \Rn\setminus B_R(0),$ and $y \in K,$ then $|x-y| \sim |x|.$ Therefore using the estimate of $G_{P}$ from Lemma~\ref{l:heat-ker} and H\"{o}lder inequality, we have 
\begin{align}\label{ext}
\int_{\Rn\setminus B_R(0)} \frac{u_{\alpha}^2(x)}{|x|^{2s}} \, {\rm d}x
&\leq \int_{\Rn\setminus B_R(0)} \frac{{\rm d}x}{|x|^{2s}} \left( \int_{\Rn} G_{P}(x, y) \, |\varphi(y)| \, {\rm d}y  \right)^2 \notag\\
& \leq C(K) \int_{\Rn\setminus B_R(0)} \frac{{\rm d}x}{|x|^{2s}} \int_{K} G_P^2(x, y) \, \varphi^2(y) \, {\rm d}y \notag \\
& \leq C  \int_{K} \int_{\Rn\setminus B_R(0)} \bigg(\frac{|x|^{-2(N-2s)} + (|x|^{-2\gamma} + |y|^{-2\gamma})|x|^{-2(N -2s - \gamma)} }{|x|^{2s}} \notag\\ 
&\qquad +\frac{|x|^{-2\gamma} |y|^{-2\gamma}|x|^{-2(N-2s - 2\gamma)}}{|x|^{2s}}\bigg)\, {\rm d}x \, {\rm d}y \notag \\
& \leq C \int_{K} \int_{\Rn\setminus B_R(0)} \Bigl( |x|^{-2(N -s)}+ |y|^{-2\ga}|x|^{-2(N-s)} \notag\\
&\qquad\qquad+ |y|^{-2 \gamma} |x|^{-2(N -s -\gamma)}  \Bigr)   \, {\rm d}x \, {\rm d}y < \infty,
\end{align}
since $N>2s$ and $\gamma \in (0, (N-2s)/2).$
For  $x \in B_R(0)$, $x\neq 0$,  using \eqref{28-5-5} and  Lemma \ref{l:heat-ker}, we have  
\begin{eqnarray*}
u_{\alpha}(x) &\leq& C\int_{K} G_{P}(x, y)\, {\rm d}y\\
&\leq&C\int_K\bigg[(x-y)^{-(N-2s)}+(|x|^{-\ga}+|y|^{-\ga})|x-y|^{-(N-2s-\ga)}\\
&&\qquad+|x|^{-\ga}|y|^{-\ga}|x-y|^{-(N-2s-2\ga)}\bigg]\, {\rm d}y\\
&\leq& C(R)(1+|x|^{-\ga})+C(1+|x|^{-\ga})\int_{K}|y|^{-\ga}|x-y|^{-(N-2s-\ga)}\, {\rm d}y\\
&\leq& {C}(R)|x|^{-\ga}\bigg(1+\int_{K}|y|^{-\ga}|x-y|^{-(N-2s-\ga)}\, {\rm d}y\bigg).
\end{eqnarray*}
Using Young's inequality we see that 
$$|y|^{-\ga}|x-y|^{-(N-2s-\ga)}\leq \frac{\ga}{N-2s}|y|^{-(N-2s)} +
\frac{N-2s-\ga}{N-2s} |x-y|^{-(N-2s)},$$
which implies, 
$$\int_{K}|y|^{-\ga}|x-y|^{-(N-2s-\ga)}\, {\rm d}y\leq C(R), \quad \text{for all}
\; x\in B_R(0),$$
that is, 
\begin{equation}\lab{30-5-2}
u_{\alpha}(x)\leq \frac{C(R)}{|x|^{\ga}}.
\end{equation}
Combining we get
\begin{eqnarray}\lab{int}
\int_{B_R(0)}\frac{u_{\alpha}^2(x)}{|x|^{2s}} \, {\rm d}x \leq C(R) \int_{B_R(0)} \frac{ {\rm d}x}{|x|^{2s + 2\gamma}} <\infty,
\end{eqnarray}
as $\gamma\in(0, ({N-2s})/{2})$ and $N>2s$. Hence \eqref{ext} and \eqref{int} proves the lemma.

\end{proof}

\medskip 

Before going further let us introduce quadratic form associated with the Hardy operator. We 
first define the quadratic form $\mathcal{E}$ of $(\Delta)^{s}:=-(-\Delta)^{s},$ in the usual way as in \cite{Fuku} (also see \cite{GGJP}) : 

\begin{equation}\lab{28-5-1}
\mathcal{E}[f] := \lim_{t \rightarrow 0+} \frac{(f - S(t)f, f )_{L^2}}{t}
= \frac{c_{N,s}}{2} \int_{\Rn} \int_{\Rn} \frac{(f(x) - f(y))^2}{|x - y|^{N + 2s}} \, {\rm d}x \, {\rm d}y,
\end{equation}
for $f \in L^{2}(\Rn)$ and $S(t)$ is the semigroup generated by $(\Delta)^s$. Moreover, the domain of the quadratic form is given by 

$$
\mathcal{D}(\mathcal{E}) := \{ f \in L^2(\Rn) : \mathcal{E}[f] < \infty \}. 
$$


Similarly, we define the quadratic form associated with  $-P$, where $P$ is the fractional Hardy operator. Recall that 
$\{ \tilde{S}(t)\}$  is the semigroup generated by $-P.$ We define 

$$
\tilde{\mathcal{E}}[f] := \lim_{t \rightarrow 0+} \frac{\big(f - \tilde{S}(t)f, f \big)_{L^2}}{t}, \quad f \in L^2(\Rn),
$$
and the domain 

$$
\mathcal{D}(\tilde{\mathcal{E}}) := \{ f \in L^2(\Rn) : \tilde{{\mathcal{E}}}[f] < \infty \}.
$$
Furthermore, we can define a bilinear form on $\mathcal{D}(\mathcal{\tilde E})$ as
follows: for $f, g\in \mathcal{D}(\mathcal{\tilde E})$
$$
\mathcal{\tilde E}(f, g)=\lim_{t\to 0+} \frac{1}{t} \big(f - \tilde S(t)f,\, g \big)_{L^2}.$$
Also note that 
\begin{equation}\label{28-5-6}\mathcal{\tilde E}(f, g)=\frac{1}{2} \left( \mathcal{\tilde E}[f+g]-\mathcal{\tilde E}[f]
-\mathcal{\tilde E}[g]\right).\end{equation}
For more details on bilinear form see \cite[Chapter 1]{Fuku}. 

This form plays a key role in our studies below. Now we recall an important lemma from \cite{GGJP} concerning the quadratic form of $\mathcal{E}$ and $\tilde{\mathcal{E}}.$

\begin{lemma}\label{lem-quad}\cite{GGJP}
We have $\mathcal{D}(\mathcal{E}) \subset \mathcal{D}(\tilde{\mathcal{E}})$ and 

$$
\tilde{\mathcal{E}}[f] = \mathcal{E}[f] - \theta\int_{\Rn} \frac{f^2(x)}{|x|^{2s}} \, {\rm d}x, \quad f \in \mathcal{D}(\mathcal{E}).
$$

\end{lemma}

\begin{remark}\lab{r:imp}
It is also follows from \cite[Proposition~5]{BDK} (also see \cite[(5.2)]{GGJP}) that 

\begin{equation}
\tilde{\mathcal{E}}[f] + \theta\int_{\Rn} \frac{f^2(x)}{|x|^{2s}} \, {\rm d}x \geq \mathcal{E}[f], 
\quad f \in L^2({\Rn}). 
\end{equation}

\end{remark}

\medskip

\subsection{Proof of Theorem \ref{t:int form}}

\begin{proof}

Recall $\{ \tilde{S}(t) := e^{-tP} \}$  is the semigroup generated by $-P.$ It is known (see \cite[Proposition~2.4]{GGJP}) that $\{ \tilde{S}(t) \}$ is a $L^2({\Rn})$ {\emph contraction semigroup}, i.e. 
$$\|\tilde{S}(t)\| \leq 1.$$ Therefore Hille-Yosida theorem states that $(\alpha I + P)$ is invertible for every 
$\alpha > 0$ and 
$$
(\alpha I + P)^{-1} = R(\alpha), \quad\text{where}\quad R(\alpha) v := \int_{0}^{\infty} e^{-\alpha \tau} \tilde{S}(\tau) \, v \, {\rm d} \tau.
$$
In particular, $R(\alpha) v \in \mathcal{D}(-P)\subset  \mathcal{D}(\tilde{\mathcal{E}})$.

Define $u_{\alpha} : = R(\alpha) \varphi.$  Then by above
 $(P + \alpha I) R(\alpha) = \mathrm{id}$ and therefore, it holds  that
$$
(P + \alpha I) u_{\alpha} = \varphi, \quad \text{in} \quad L^2(\Rn).
$$  
Thus,
\begin{equation}\lab{28-5-4}
\big( (P + \alpha I)u_{\alpha}, f \big)_{L^2(\Rn)} = \int_{\Rn} \varphi  \, f \, {\rm d} x \quad\forall\, f\in L^2(\Rn). 
\end{equation}
By Lemma \ref{L2-hardy}, we have 
$$\displaystyle\int_{\Rn} \frac{u_{\alpha}^2(x)}{|x|^{2s}} \, {\rm d}x < \infty$$ and $R(\alpha) \varphi\in  \mathcal{D}(\tilde{\mathcal{E}})$ also implies $\mathcal{\tilde E}(u_\alpha)<\infty$. Therefore, applying Remark \ref{r:imp}, we obtain $u_\alpha\in \mathcal{D}(\mathcal{E})$. Hence $u_\alpha\in{H}^s(\Rn)$ for each $\alpha>0$. Thus, by Sobolev inequality $u_\alpha\in\dot{H}^s(\Rn)$ for each $\alpha>0$.


 Moreover, from \eqref{3-6-1} it is easy to see that
$$(P u_\alpha, f)_{L^2(\Rn)}= \mathcal{\tilde E}( u_\alpha, f)   \quad\text{for any} \quad f\in \mathcal{D}(\mathcal{E}).$$ By \eqref{28-5-6}, we know
$$\mathcal{\tilde E}( u_\alpha, f)=\frac{1}{2} \left( \mathcal{\tilde E}[u_\alpha+f]-\mathcal{\tilde E}[f]
-\mathcal{\tilde E}[u_\alpha]\right).$$
Combining this with Lemma \ref{lem-quad} yields
$$(P u_\alpha, f)_{L^2(\Rn)}=\frac{c_{N,s}}{2}\int_{\Rn} \int_{\Rn} \frac{(u_{\alpha}(x) - u_{\alpha}(y))(f(x)-f(y))}{|x - y|^{N + 2s}} \, {\rm d}x \, {\rm d}y - 
\theta \int_{\Rn} \frac{u_{\alpha}f}{|x|^{2s}} \, {\rm d}x. $$

%

%
%
%

In particular, taking $f=u_\alpha$ in \eqref{28-5-4} and defining $K := \mathrm{supp}(\varphi)$, it follows
\begin{align*}
&&  \frac{c_{N,s}}{2}\int_{\Rn} \int_{\Rn} \frac{|u_{\alpha}(x) - u_{\alpha}(y)|^2}{|x - y|^{N + 2s}} \, {\rm d}x \, {\rm d}y - 
\theta \int_{\Rn} \frac{u_{\alpha}^2}{|x|^{2s}} \, {\rm d}x \, + \, 
\alpha \int_{\Rn} u_{\alpha}^2 \, {\rm d}x \\
&&\qquad= \int_{K}|x|^s\varphi\frac{u_\alpha}{|x|^s} \, {\rm d}x
\leq\tilde{C}(\varepsilon,K) \int_{K}   \varphi^2 \, {\rm d}x \,  + \, \varepsilon C(K) \int_{\Rn}  \frac{u_{\alpha}^2}{|x|^{2s}} \, {\rm d}x.
\end{align*}
Since $u_\alpha\in \dot{H}^s(\Rn)$, choosing $\eps>0$ such that $\theta+\eps<\Lambda_{N,s}$ and using Remark \ref{r:1}, we get  $\{ u_{\alpha} \}$ is uniformly bounded in $\dot{H}^{s}(\Rn)$ and $\|\sqrt{\alpha}u_\alpha\|_{L^2(\Rn)}$ is uniformly bounded. Hence there exists $u\in \dot{H}^s(\Rn)$ and $g\in L^2(\Rn)$ such that up to 
a subsequence, 
$$u_{\alpha{_k}} \rightharpoonup u\quad\text{in}\quad\dot{H}^s(\Rn)\quad\text{and}\quad u_{\alpha{_k}} \rightarrow u \quad a.e, $$ and 
$$\sqrt{\alpha_k}u_{\alpha_k}  \rightharpoonup g \quad\text{in}\quad L^2(\Rn)$$
 as $\alpha_k \rightarrow 0$. This implies, for $f\in C^\infty_c(\Rn)$ 
 $$\alpha\int_{\Rn}u_\alpha f \, {\rm d}x=\sqrt{\alpha}\big(\sqrt{\alpha}u_\alpha, f \big)_{L^2(\Rn)}\longrightarrow 0 \quad\text{as}\quad \alpha\to 0.$$
 Therefore, from \eqref{28-5-4}, it is easy see that 
$u$ satisfies $Pu=\varphi$ in the weak sense. Since
$$
u_{\alpha} = R(\alpha) \phi = \int_{0}^{\infty} e^{-\alpha \tau} \tilde{S}(\tau) \, \varphi \, {\rm d}\tau,
$$
using dominated convergence theorem and letting $\alpha \rightarrow 0$, we obtain 
$$
u_{\alpha} \rightarrow \tilde{u} : = \int_{0}^{\infty} \tilde{S}(\tau) \, \varphi \, {\rm d}\tau,\quad x\neq 0. 
$$ 
 (Here we have used $\tilde{S}(\tau) \, |\varphi|$ as the dominant function. Using Tonelli's theorem in the definition of $\tilde{S}(\tau) \, |\varphi|$, and arguing as in the proof of \eqref{30-5-2}, it follows that $$\int_0^\infty\tilde{S}(\tau) \, |\varphi(x)| \, {\rm d}\tau < \infty \quad \text{for all}\quad x\neq 0.$$ Hence $\tilde{S}(\tau) \, |\varphi|\in L^1(0,\infty)$.)

We also know from the definition of $\tilde{S}(\tau)\varphi$ that
$$\tilde{S}(\tau)\varphi=\int_{\Rn}p(t,x,y) \, \varphi(y) \, {\rm d}y,$$
where $p(.,.,.)$ is the heat kernel of $P$ as described in \eqref{heat-ker}. Therefore, using Fubini 
$$\int_{0}^{\infty} \tilde{S}(\tau) \, \phi \, {\rm d}\tau =\int_{\Rn} G_p(x, y) \, \phi \, {\rm d}y=\psi,\quad x\neq 0.$$
Hence $\psi = \tilde u$.  This implies $u_{\alpha} \to \psi$ a.e.. On the other hand, since we already proved $u_{\alpha} \to u$ a.e., it implies $u=\psi$ a.e. Hence it holds $P\psi=\phi$  in the weak sense.
\end{proof}

{\bf Proof of Corollary \ref{thm-1}:}\begin{proof}
From Theorem~\ref{t:int form} we know that $\psi:=\displaystyle\int_{\Rn} G_{P}(x, y) \varphi(y) \, {\rm d}y$ is a weak solution of ~\eqref{main-semi}. Moreover, from Remark \ref{r:1}, we see that weak solution of \eqref{main-semi} is unique. Therefore we get the conclusion. 
\end{proof}

\vspace{3mm}

{\bf Proof of Theorem \ref{l:fund-sol}}\begin{proof}
Let $x\neq 0$. To prove the theorem we need to show that
$$\int_{\Rn}G_P(x, z) P(f)(z) \, {\rm d}z=f(x) \quad \forall\, f\in C^\infty_c(\Rn).$$

Consider $\phi_n\in C^\infty_c(\Rn)$ such that $\phi_n\to\de_x$ in the sense of distributions. Furthermore, we can choose that $\phi_n\geq 0$ and  $\mathrm{supp}(\phi_n)\subset B_\kappa (x)$ for all $n$, where $\kappa<\frac{|x|}{2}$. 
Corresponding to $\phi_n$, we define 
$$\psi_n (z) := \displaystyle\int_{\Rn}   G_P(z, y)\phi_n (y) \, {\rm d}y,$$  where $G_P$ is the  Green function of $P$ as defined in \eqref{Green}. By Theorem \ref{t:int form}, $\psi_n\in \dot{H}^s(\Rn)$ and 
\begin{equation}\lab{28-5-8}
 \frac{c_{N,s}}{2}\int_{\Rn} \int_{\Rn} \frac{(\psi_n(z) - \psi_n(y))(f(z)-f(y))}{|z - y|^{N + 2s}} \, {\rm d}z\,{\rm d}y - 
\theta \int_{\Rn} \frac{\psi_nf}{|z|^{2s}} \, {\rm d}z=\int_{\Rn}\phi_n f {\rm d}z , 
\end{equation}
for all $ f\in C^\infty_c(\Rn)$.
Therefore,
\be\lab{6-6-1}\int_{\Rn}\psi_n(-\De)^sf \, {\rm d}z-\theta \int_{\Rn} \frac{\psi_nf}{|z|^{2s}} \, {\rm d}z=\int_{\Rn}\phi_n f \, {\rm d}z \quad \forall\, f\in C^\infty_c(\Rn), \ee

i.e. \begin{equation}\lab{30-5-1}\int_{\Rn}\psi_n(z) P(f)(z)\, {\rm d}z=\int_{\Rn}\phi_n(z) f(z)\, {\rm d}z \quad \forall\, f\in C^\infty_c(\Rn).\end{equation}
We point out that in order to use the integration by parts formula  in \eqref{6-6-1},   we argue by density since $\psi_n\in \dot{H}^s(\Rn)$ and we use the fact that  $f\in C^\infty_c(\Rn)$, which also implies  
$$|(-\Delta)^s f(x)|\leq \frac{C}{1+|x|^{N+2s}}.$$

Since the heat kernel $p(t, z,y)$ is continuous in $z$ (see \cite[Lemma 4.10]{GGJP}), it is easy to check that $z\mapsto G_P(z,y)$ is continuous
for $z\neq y$. Therefore,  $\psi_n(z)\to G_P(z, x)$ pointwise. Clearly, RHS of \eqref{30-5-1}$\to f(x)$ as $n\to\infty$. 

\medskip

{\bf Claim:} If $f\in  C^\infty_c(\Rn)$, then taking  the limit $n\to\infty$ in \eqref{30-5-1} yields 
$$\int_{\Rn}G_P(z,x) P(f)(z)\, {\rm d}z=f(x) .$$
 First we note that for $f\in C^\infty_c(\Rn)$, it is easy to check 
\begin{equation}\lab{30-5-4}| P f(z)| \leq C\left(\frac{1}{1+ |z|^{N+2s}}+\frac{1}{|z|^{2s}}\mathcal{\chi}_{\mathrm{supp}(f)} \right). 
\end{equation}
We split the proof of the claim in two steps.

\medskip

\noindent{\bf Step 1.} We show that for each $n$ 
\begin{equation*}
\int_{y\in B_\kappa(x)}\int_{\Rn}\phi_n(y)G_P(z,y) |P(f)(z)|\, {\rm d}z \, {\rm d}y<\infty.
\end{equation*}
Using Tonelli's theorem, we get
\begin{eqnarray}\lab{5-6-1}
&&\int_{y\in B_\kappa(x)}\int_{\Rn}\phi_n(y)G_P(z,y) |P(f)(z)|\, {\rm d}z\, {\rm d}y \no\\
&\leq&C_n\int_{\Rn}\bigg(\int_{y\in B_\kappa(x)}G_P(z,y)dy\bigg) |P(f)(z)| \, {\rm d}z \no\\
&=&C_n\int_{B_R(0)}\bigg(\int_{y\in B_\kappa(x)}G_P(z,y)dy\bigg) |P(f)(z)| \, {\rm d}z \no\\
&&\quad+
C_n\int_{\Rn\setminus B_R(0)}\bigg(\int_{y\in B_\kappa(x)}G_P(z,y)dy\bigg) |P(f)(z)| \, {\rm d}z.
\end{eqnarray}
Now for  $R>2|x|$, repeating an estimate as in the derivation of \eqref{30-5-2} we see that
$$\int_{B_R(0)}\bigg(\int_{y\in B_\kappa(x)}G_P(z,y)dy\bigg) |P(f)(z)| \, {\rm d}z\leq C\int_{B_R(0)}\frac{|P(f)(z)|}{|z|^\gamma}\, {\rm d}z.$$ Also it is easy to see that for the choice of $R$ and $z\in\Rn\setminus B_R(0)$, it holds
$$\bigg|\int_{y\in B_\kappa(x)}G_P(z,y)\, {\rm d}y\bigg|\leq C,$$ 
where the constant $C$ does not depend on $z$. Therefore from \eqref{5-6-1} and also using \eqref{30-5-4},  we obtain
\begin{eqnarray*}
\int_{y\in B_\kappa(x)}\int_{\Rn}\phi_n(y)G_P(z,y) |P(f)(z)| \, {\rm d}z \, {\rm d}y
&\leq& C\bigg(\int_{B_R(0)}\frac{{\rm d}z}{|z|^{\gamma+2s}}+\int_{\Rn\setminus B_R(0)}\frac{{\rm d}z}{|z|^{N+2s}}\bigg)\\
&<&\infty.
\end{eqnarray*}
This proves Step 1.

\medskip

\noindent{\bf Step 2.} We complete the proof of the claim in this step.
Using Step 1 and Fubini's theorem we see that 
$$\int_{\Rn}\psi_n(z) P(f)(z) \, {\rm d}z=\int_{\Rn} \phi_n(y) \left(\int_{\Rn}G_P(z,y) P(f)(z) \, {\rm d}z\right) {\rm d}y.$$
We recall that, $G_P(z,y)=G_P(y,z)$. Therefore, if we can show that 
$$y\longmapsto \int_{\Rn}G_P(z,y) P(f)(z)\, {\rm d}z=\int_{\Rn}G_P(y, z) P(f)(z) \, {\rm d}z,$$
is continuous at $y=x\neq 0$, then we can conclude 
$$\lim_{n\to\infty}\int_{\Rn} \phi_n(y) \left(\int_{\Rn}G_P(z,y) P(f)(z) \, {\rm d}z\right)\, {\rm d}y=\int_{\Rn}G_P(z,x) P(f)(z) \, {\rm d}z.$$
This would complete the proof of the claim.

\vspace{2mm}

Hence, we are left to show that $y\longmapsto \displaystyle\int_{\Rn}G_P(z,y) P(f)(z) \, {\rm d}z$ is continuous at $y=x\neq 0$.

\vspace{2mm}

To see this, first we observe that as $x\neq 0$ and $\ga\in (0,(N-2s)/{2})$, Lemma \ref{l:heat-ker} implies given $\eps>0$ there exists $\de,\, r_0\in(0, \frac{|x|}{3})$ , such that
\begin{equation}\lab{30-5-3}
\int_{B_{\de}(y)} G_P(y,z) \, {\rm d}z<\eps \quad\forall\, y\in B_{r_0}(x).
\end{equation}
Now choose $r_1<\min\{r_0, \frac{\de}{2}\}$. Then we have 
\be\lab{5-6-4} B_\frac{\de}{2}(x)\subset B_{\de}(y)  \quad\forall\, y\in B_{r_1}(x).\ee
Also, choose $\delta_1< \frac{|x|}{3}$ satisfying 

\be\lab{5-6-3} \displaystyle\int_{B_{\de_1}(0)}\frac{{\rm d}z}{|z|^{\ga+2s}}<\eps.\ee 


Now we write 
\begin{eqnarray}
\int_{\Rn}G_P(z,y) P(f)(z) \, {\rm d}z&=& \int_{\Rn\setminus \big( B_\frac{\de}{2}(x)\cup B_{\de_1}(0)\big)}G_P(z,y) P(f)(z) \, {\rm d}z \no\\
&&\quad+ \int_{B_\frac{\de}{2}(x)}G_P(z,y) P(f)(z) \, {\rm d}z \no\\
&&\qquad\quad + \int_{B_{\de_1}(0)}G_P(z,y) P(f)(z) \, {\rm d}z\no.
\end{eqnarray}
As we are interested in $y\to x$, let us assume $y\in B_\kappa(x)$. Also, by our choice,
$$B_{\de_1}(0)\cap B_\kappa(x)=\emptyset.$$ Thus using Lemma \ref{l:heat-ker}, \eqref{30-5-4} and \eqref{5-6-3} it follows that
\begin{equation}\label{eq27th June 1}
\left |\int_{B_{\de_1}(0)}G_P(z,y) P(f)(z) \, {\rm d}z\right|\leq C\int_{B_{\de_1}(0)}\frac{{\rm d}z}{|z|^{\ga+2s}}=O(\eps)\end{equation}
and
\begin{equation}\label{eq27th June 2}
\left |\int_{B_{\de_1}(0)}G_P(z,x) P(f)(z) \, {\rm d}z\right|\leq C\int_{B_{\de_1}(0)}\frac{{\rm d}z}{|z|^{\ga+2s}}=O(\eps).\end{equation}
Next, we observe from \eqref{30-5-4} that for $z\in B_\frac{\de}{2}(x)$, $|Pf(z)|\leq C,$ where $C$ is independent of~$z$. Therefore, from \eqref{5-6-4} and \eqref{30-5-3}, it follows
\begin{equation}\label{eq27th June 3}
\left| \int_{B_\frac{\de}{2}(x)}G_P(z,y) P(f)(z) \, {\rm d}z\right|< C \eps.
 \end{equation}
Further, by
dominated convergence theorem 
\begin{equation}\label{eq27th June 4}
\int_{\Rn\setminus \big( B_\frac{\de}{2}(x)\cup B_{\de_1}(0)\big)}G_P(z,y) P(f)(z) \, {\rm d}z
\xrightarrow{y\to x}
\int_{\Rn\setminus \big( B_\frac{\de}{2}(x)\cup B_{\de_1}(0)\big)}G_P(z,x) P(f)(z)\, {\rm d}z,
\end{equation}
where we use \eqref{30-5-4}. Therefore, using \eqref{30-5-3},  \eqref{eq27th June 1},  \eqref{eq27th June 2},   \eqref{eq27th June 3} and \eqref{eq27th June 4},   there exists a neighborhood of $x$ such that 
$$\left |\int_{\Rn}G_P(z,y) P(f)(z) \, {\rm d}z-\int_{\Rn}G_P(z,x) P(f)(z) \, {\rm d}z\right|\leq C \eps.$$
This completes the proof of the claim as $\eps$ is arbitrary.
\end{proof}

\medskip

{\bf Proof of Theorem \ref{t:int-gen}}:

\begin{proof}
Let $0\leq \varphi\in L^\frac{2N}{N+2s}(\Rn)$ and $\varphi$ be lower semi-continuous. Therefore, there exists a sequence of bounded continuous functions
$\varphi_n$ satisfying $0\leq \varphi_1\leq \varphi_2\leq \cdots\leq \varphi_n\leq \cdots \varphi$ and
$$\lim_{n\to\infty} \varphi_n(x)=\varphi(x),\quad \forall\, x.$$ Furthermore, by multiplying with 
suitable cut-off functions we can also assume that $\varphi_n\in C_c(\Rn)$.

Let 
\be\lab{16-6-5}\psi_n(x)=\int_{\Rn} G_p(x, y) \varphi_n(y) \, {\rm d}y, \quad x\neq 0.\ee
Then by Theorem \ref{t:int form}, $\psi_n\in \dot{H}^s(\Rn)$ and satisfies
\be\lab{14-6-1}(-\De)^s \psi_n -\theta\frac{\psi_n}{|x|^{2s}}= \varphi_n\quad\text{in }\quad \Rn,\ee
in the weak sense.
Since $\|\varphi_n\|_{L^\frac{2N}{N+2s}(\Rn)}\leq \|\varphi\|_{L^\frac{2N}{N+2s}(\Rn)}$,  taking $\psi_n$ as a test function in \eqref{14-6-1}, we obtain
\begin{eqnarray*}\frac{c_{N,s}}{2}\int_{\Rn}\int_{\Rn}\frac{|\psi_n(x)-\psi_n(y)|^2}{|x-y|^{N+2s}}\, {\rm d}x\, {\rm d}y\, - \,  \theta \, \int_{\Rn} \frac{|\psi_n|^2}{|x|^{2s}}\, {\rm d}x&=&\int_{\Rn}\varphi_n\psi_n \, {\rm d}x\\
&\leq& \|\varphi_n\|_{L^\frac{2N}{N+2s}(\Rn)}\|\psi_n\|_{L^{2^*}(\Rn)}\\
&\leq& \frac{C}{S}\|\psi_n\|_{\dot{H}^s(\Rn)},
\end{eqnarray*}
where $S$ is the Sobolev constant. Using Remark \ref{r:1} on the LHS of above expressions yields $\|\psi_n\|_{\dot{H}^s(\Rn)}\leq \tilde C\quad\forall\, n$. Therefore, there exists $\psi\in\dot{H}^s(\Rn)$ such that $\psi_n\rightharpoonup \psi$ in 
$\dot{H}^s(\Rn)$ and $\psi_n\to\psi$ a.e. From \eqref{14-6-1}, we also have
\begin{equation}\lab{14-6-2}
 \frac{c_{N,s}}{2}\int_{\Rn} \int_{\Rn} \frac{(\psi_n(z) - \psi_n(y))(f(z)-f(y))}{|x - y|^{N + 2s}} \, {\rm d}z\,{\rm d}y - 
\theta \int_{\Rn} \frac{\psi_nf}{|z|^{2s}} \, {\rm d}z=\int_{\Rn}\varphi_n f \, {\rm d}z , 
\end{equation}
for all $ f\in C^\infty_c(\Rn)$. Now we take the limit $n\to\infty$ on both sides of the above expressions. For the 2nd term on the LHS we can pass the limit inside the integral sign using Vitali's convergence theorem via Hardy inequality. Hence $\psi$ satisfies
\eqref{main-semi}. By uniqueness of weak solution, \eqref{main-semi} has unique solution as $\theta<\Lambda_{N,s}$. Further, from \eqref{16-6-5} we have
for $x\neq 0$,
$$\psi(x)=\lim_{n\to\infty}\int_{\Rn}G_P(x,y)\varphi_n(y) \, {\rm d}y.$$
Using \eqref{eq:27thJune}, applying Lebesgue monotone convergence theorem the above expression yields
$$\psi(x)= \int_{\Rn} G_p(x, y) \varphi(y) \, {\rm d}y, \quad  x\neq 0.$$
 This completes the proof.

\end{proof}

\medskip

{\bf Proof of Theorem \ref{t:unique}}: 
\begin{proof}
Let $\Phi$ satisfy
\be\lab{17-6-1}P \Phi(x, \cdot)=\delta_x\quad \text{in}\quad \mathbb{R}^N, \quad x\neq 0,\ee
in the sense of distribution and 
\be\lab{17-6-2}\psi_{\varphi}:=\int_{\Rn}\Phi(x,y)\varphi(y)\, {\rm d}y\in\dot{H}^s(\Rn)\quad\forall\quad 0\leq \varphi\in C_c^\infty(\Rn).\ee
Let $f\in C_c^\infty(\Rn)$ be arbitrary and we further choose $0\leq \varphi\in C_c^\infty(\Rn)$ arbitrarily, then from \eqref{17-6-1} we have
$$\int_{\Rn}\Phi(x,y)Pf(y)\, {\rm d}y=f(x), \quad x\neq 0.$$ Consequently,
\be\lab{17-6-3}
\int_{\Rn}\varphi(x)\bigg(\int_{\Rn}\Phi(x,y)Pf(y)\, {\rm d}y\bigg)\, {\rm d}x=\int_{\Rn}f\varphi\, {\rm d}x.
\ee

\medskip

{\bf Claim}: \be\lab{17-6-7}\int_{\Rn}\varphi(x)\bigg(\int_{\Rn}\Phi(x,y)Pf(y) \, {\rm d}y\bigg)\, {\rm d}x
=\int_{\Rn}Pf(y)\bigg(\int_{\Rn}\Phi(x,y)\varphi(x)\, {\rm d}x\bigg)\, {\rm d}y.\ee

Assuming the claim, first we complete the proof. Using the claim, we obtain from \eqref{17-6-3} that
\begin{eqnarray*}
\int_{\Rn}f\varphi\, {\rm d}x&=&\int_{\Rn}Pf(y)\psi_{\varphi}(y)\, {\rm d}y\\
&=&\int_{\Rn}\psi_{\varphi}(-\De)^sf \, {\rm d}y-\theta\int_{\Rn}\frac{f\psi_{\varphi}}{|y|^{2s}}\, {\rm d}y.
\end{eqnarray*}
Since $\psi_{\varphi}\in\dot{H}^s(\Rn)$ and $f\in C_c^\infty(\Rn)$, using an argument as in \eqref{28-5-8} --\eqref{30-5-1}  (see the line just after \eqref{30-5-1}) we obtain
$$\int_{\Rn}f\varphi\, {\rm d}x = \int_{\Rn}(-\De)^\frac{s}{2}\psi_{\varphi}(-\De)^\frac{s}{2}f \, {\rm d}y-\theta\int_{\Rn}\frac{f\psi_{\varphi}}{|y|^{2s}}\, {\rm d}y. $$
Since in the above expression $f\in C_c^\infty(\Rn)$ is arbitrary, we conclude 
 $\psi_{\varphi}$ satisfies
$$(-\De)^s\psi_{\varphi}-\theta\frac{\psi_{\varphi}}{|x|^{2s}}=\varphi,$$ in the weak sense. Therefore, using Corollary \ref{thm-1}, we have
\be\lab{17-6-4}
\psi_{\varphi}(x)=\int_{\Rn}G_P(x,y)\varphi(y)\, {\rm d}y, \quad x\neq 0.
\ee
Combining \eqref{17-6-4} with \eqref{17-6-2} yields
\be\lab{17-6-6}\int_{\Rn}\bigg(G_P(x,y)-\Phi(x,y)\bigg)\varphi(y)\, {\rm d}y=0, \quad x\neq 0.\ee
Since $\varphi\geq 0$ is an arbitrary $C_c^\infty$ function in $\Rn$ and
since $G_P$ and $\Phi$ are continuous functions in $y$ for $y\neq x$ and $y\neq 0$ and in $L^1_{loc}(\Rn)$ a.e. $x$ , we see from \eqref{17-6-6}
and a density argument that
$$\int_{\Rn}\bigg(G_P(x,y)-\Phi(x,y)\bigg)\varphi(y)\, {\rm d}y=0, \quad \text{for all}\;\; \varphi\in C_c(\Rn),$$
which in turn, implies
$G_P(x,\cdot)=\Phi(x,\cdot)$ in $\Rn\setminus\{x\}$. Since $x\neq 0$ is also arbitrary, $G_P=\Phi$ in $(\Rn\setminus\{0\})\times (\Rn\setminus\{0\})$.

\medskip

Therefore, we are left to prove the claim \eqref{17-6-7}. Since $\Phi$ and $\varphi$ are nonnegative functions, in order to justify that claim using Fubini, it's enough to show that $$\displaystyle\iint_{R^{2N}}|Pf(y)|\Phi(x,y)\varphi(x) \, {\rm d}(x\otimes y)<\infty.$$ Using Tonelli's theorem and \eqref{30-5-4}, we estimate the above integration as below

\begin{eqnarray}\lab{17-6-8}
\displaystyle\iint_{R^{2N}}|Pf(y)|\Phi(x,y)\varphi(x) \, {\rm d}(x\otimes y)&=&\int_{\Rn}|Pf(y)|\bigg(\int_{\Rn}\Phi(x,y)\varphi(x)\, {\rm d}x\bigg)\, {\rm d}y\no\\
&=&\int_{\Rn}|Pf(y)|\psi_{\varphi}(y)\, {\rm d}y\no\\
&\leq&C\int_{\Rn}\left(\frac{1}{1+ |y|^{N+2s}}+\frac{1}{|y|^{2s}}\mathcal{\chi}_{\mathrm{supp}(f)} \right)\psi_{\varphi}(y)\, {\rm d}y.\no\\
&&
\end{eqnarray}
Since $\psi_{\varphi}\in\dot{H}^s(\Rn)$, using H\"{o}lder inequality
$$\int_{\mathrm{supp}(f)}\frac{\psi_{\varphi}(y)}{|y|^{2s}}\, {\rm d}y<\bigg(\int_{\Rn}\frac{|\psi_{\varphi}(y)|^2}{|y|^{2s}}\, {\rm d}y\bigg)^\frac{1}{2}
\bigg(\int_{\mathrm{supp}(f)}\frac{{\rm d}y}{|y|^{2s}}\bigg)^\frac{1}{2}<\infty.$$
On the other hand, since $\frac{1}{1+ |y|^{N+2s}}\in L^\frac{2N}{N+2s}(\Rn)$ and  $\psi_{\varphi}\in\dot{H}^s(\Rn)$ implies $\psi_{\varphi}\in L^{2^*}_s(\Rn)$, using
H\"{o}lder inequality, we immediately get that 1st integral on the RHS of \eqref{17-6-8} is also finite. Hence the claim follows.

This completes the proof of the theorem. 
  
\end{proof}

\medskip

{\bf Acknowledgement}: The research of A.~Biswas is partially supported by an INSPIRE faculty fellowship (IFA13/MA-32)
and DST-SERB grants EMR/2016/004810, \\
MTR/2018/000028.
M.~Bhakta is partially supported by the INSPIRE faculty fellowship IFA13/MA-33 and SERB MATRICS grant. D.~Ganguly is partially supported by INSPIRE faculty fellowship (IFA17-MA98).


\begin{thebibliography}{XX}

 \bibitem{AABP}{B.~ Abdellaoui, A.~Attar, R.~Bentifour, 
              I.~Peral},
 {\em On a fractional quasilinear parabolic problem: the influence of the {H}ardy potential},
   {NoDEA Nonlinear Differential Equations Appl.},
   {25}(4), Art. 30, 34,  {2018}.


\bibitem{AADP}{B.~ Abdellaoui, A.~Attar, A.~Dieb, I.~Peral},
 {\em Attainability of the fractional {H}ardy constant with nonlocal
              mixed boundary conditions: applications}, {Discrete Contin. Dyn. Syst.},
  {38}(12), 5963--5991, 2018.
  
\bibitem{AMP}{B.~ Abdellaoui, M. Medina,  I.~Peral, A. Primo}, {\em Optimal results for the fractional heat equation involving the
              {H}ardy potential}, {Nonlinear Anal.},
 {140}, {166--207}, {2016}.

\bibitem{AMP1}{B.~ Abdellaoui, M. Medina,  I.~Peral, A. Primo}, {\em The effect of the {H}ardy potential in some {C}alder\'{o}n-{Z}ygmund properties for the fractional {L}aplacian}, {J. Differential Equations},
{260}(11), {8160--8206} {2016}.



\bibitem{BGM} M.~Bhakta, D.~Ganguly, L.~Montoro; \emph{Fractional Hardy equations with critical and supercritical exponents}, preprint. 




\bibitem{BDK} K.~Bogdan,  B.~Dyda, P.~ Kim, \emph{Hardy inequalities and non-explosion results for semigroups}, Potential Anal. 44 (2016), no. 2, 229--247.


\bibitem{GGJP} K.~Bogdan, T.~Grzywny, T.~Jakubowski, D.~Pilarczyk,  \emph{ Fractional Laplacian with Hardy Potential},  Comm. Partial Differential Equations, 44(1), 20-50, 2019.
    
\bibitem{B}{C.~Bucur},
\emph{Some observations on the {G}reen function for the ball in the
              fractional {L}aplace framework}, {Commun. Pure Appl. Anal.}, 15(2), 657-699, 2016.
      

\bibitem{CDM} G. Caristi, L. D'Ambrosio, Lorenzo and E. Mitidieri. Representation formulae for solutions to
some classes of higher order systems and related Liouville theorems. Milan J. Math. 76 (2008),
27--67.

\bibitem{CLO} {\sc W.~Chen, C.~ Li, B.~Ou}; Classification of solutions for an integral equation, \textit{Comm. Pure Appl. Math.} 59 (2006), no. 3, 330--343.


\bibitem{DMPS}{\sc  Dipierro,~S.; Montoro,~L.; Peral,~I.; Sciunzi,~B.}, Qualitative properties of positive solutions to nonlocal critical problems involving the Hardy-Leray potential. \textit{Calc. Var. Partial Differential Equations} 55 (2016), no. 4, Art. 99, 29 pp.


\bibitem{FF1} {\sc M.~.M.~Fall}, \emph{Semilinear elliptic equations for the fractional Laplacian with Hardy potential} To appear in  Nonlinear Analysis, 2018.

\bibitem{FF}{M.~.M.~Fall, V.~Felli},
{\em Unique continuation properties for relativistic {S}chr\"{o}dinger
              operators with a singular potential},
{Discrete Contin. Dyn. Syst.},
{35}(12), {5827--5867}, 2015.
     

\bibitem{FLS-1} {\sc R.~L.~Frank, E.~H.~Lieb, R.~Seiringer}; \emph{Hardy-Lieb-Thirring inequalities for fractional Schr\"odinger operators}. J. Amer. Math. Soc. 21 (2008), no. 4, 925--950. 

\bibitem{FLS}
{\sc R.~L.~Frank, E.~Lenzmann, L.~Silvestre}; \emph{Uniqueness of Radial Solutions for the Fractional Laplacian}.
Communications on Pure and Applied Mathematics 19 (2016), 1671--1726.

\bibitem{Fuku}{\sc M.~Fukushima, Y.~Oshima, M.~Takeda}; \emph{Dirichlet forms and symmetric Markov processes}, volume 19 of De Gruyter Studies in Mathematics. Walter de Gruyter \& Co., Berlin, extended edition,
2011.

\bibitem{GR}{\sc N.~Ghoussoub, S.~Mazumdar, F.~Robert}; \emph{Multiplicity and stability of the Pohozaev obstruction for Hardy-Schr\"{o}dinger equations with boundary singularity}, arXiv:1904.00087v2.

\bibitem{MNS} {\sc G.~Metafune,  L.~Negro,  C.~Spina};
\emph{Sharp kernel estimates for elliptic operators with second-order discontinuous coefficients}, J. Evol. Equ. 18 (2018), no. 2, 467--514. 

\bibitem{S} {\sc L.~Silvestre}; \emph{Regularity of the obstacle problem for a fractional power of the Laplace operator}. Comm. Pure Appl. Math. 60 (2007), no. 1, 67--112.

\bibitem{Lei} {\sc Y.~Lei};  \emph{Asymptotic properties of positive solutions of the Hardy-Sobolev type equations}, J. Differential Equations, 254 (2013), no. 4, 1774--1799.

\end{thebibliography}
\end{document}